\documentclass[11pt,openany,leqno]{article}
\usepackage[colorlinks=true]{hyperref}
\usepackage{amsmath,amsthm,amsfonts,amssymb,amscd,hyperref}
\usepackage[capitalize]{cleveref}
\usepackage{graphicx}
\usepackage{soul}

\usepackage{graphics,xcolor,url}
\numberwithin{equation}{section}
\input{xy} \xyoption{all}
\usepackage[latin1]{inputenc}
\usepackage{enumerate}

\usepackage{cases}
\usepackage{hyperref}
\usepackage{epsfig} 
\input{xy} \xyoption{all} 
\usepackage{mathrsfs}

\def \exp{\mathrm{exp}}

\def \cP{{\mathscr P}}

\def\qand{\quad \text{and}\quad}

\def\Ham{\mathrm{Ham}}

\def\C{\mathbb C}
\def\R{\mathbb R}

\def\T{\mathbb T}

\def\Z{\mathbb Z}

\def\Re{\mathrm{Re}}

\def\cal{\mathcal }

\newtheorem{proposition}{Proposition}[section]

\newtheorem*{theorem*}{Theorem}
\newtheorem{corollary}[proposition]{Corollary}

\newtheorem*{problem*}{Problem}

\newtheorem{lemma}[proposition] {Lemma}

\theoremstyle{remark}

\newtheorem{remark}[proposition]{Remark}

\makeatletter

\begin{document}
\title{Generators of groups of Hamitonian maps}
\author{Pierre Berger\thanks{IMJ-PRG, CNRS, Sorbonne University, Paris University, partially supported by the ERC project 818737 Emergence of wild differentiable dynamical systems.
} and Dmitry Turaev}

\date{\today}
\maketitle
\begin{abstract}
We prove that analytic Hamiltonian dynamics on tori, annuli, or Euclidean space can be approximated by a composition of nonlinear shear maps where each of the shears depends only on the position or only on the momentum.
\end{abstract}
\section{Statement of the result.}
Let $\T^n:= \R^n/\Z^n$ denote the $n$-torus. We endow the $2n$ torus $V:= \T^n\times \T^n$ with the canonical coordinates $(q,p)=(q_1,\dots,q_n,p_1,\dots,p_n)$ and symplectic form $\omega=\sum_i dp_i\wedge dq_i$.
For a function  $H\in C^\infty (V,\R)$, the system of differential equations defined by the Hamiltonian $H$ is given~by
\begin{equation}\label{hameq1}
\dot q = \partial_p H,\qquad \dot p = - \partial_q H.
\end{equation}
The corresponding vector field
$$X_H:= (\partial_{p_1} H,\dots, \partial_{p_n} H, -\partial_{q_1} H, \dots, -\partial_{q_n} H),$$
satisfies $\omega(X_H,\cdot )= dH$;  it is called the {\em symplectic gradient} of $H$. 
The symplectic gradient defines the \emph{Hamiltonian flow} denoted by $\phi^t_H$, the family of time-$t$ maps along the trajectories of system(\ref{hameq1}). 
Similarly, given a continuous family of functions $H_t\in C^\infty (V,\R)$, $t\geq 0$, one defines the {\em time-dependent Hamiltonian system} 
\begin{equation}\label{nona}
(\dot q, \dot p) = X_{H,t} = (\partial_{p} H_t, \; - \partial_{p} H_t).
\end{equation}
The trajectories of this system define the family of maps $\phi^{s,t}_H$, $0\le s\leq t$: the solution with the initial condition 
$(q,p)$ at time $s$ arrives at the point $\phi^{s,t}_H(q,p)$ at time $t$.  Such maps preserve the symplectic form $\omega$.
The family of these maps is called the {\em non-autonomous Hamiltonian flow} of $H:=(H_t)_{t\geq 0}$. 

A symplectic map is called {\em a Hamiltonian map} if it is the map $\phi^{0,1}_H$ for a time-dependent Hamiltonian $H$. We consider the spaces $\Ham^\infty (V)$ of Hamiltonian $C^\infty$-diffeomorphisms and
$\Ham^\omega (V)$
of diffeomorphisms defined by real-analytic Hamiltonians $H_t$ which depend on $t$ continuously
in $C^\omega(V,\R)$.

Recall that the base of the $C^\omega$-topology (the inductive limit topology) on space of real-analytic functions is a collection, taken over all neighborhoods of
$V= \R^{2n} /\Z^{2n}$ in its complexification $\C^{2n}/\Z^{2n}$, of $C^0$-open sets of holomorphic functions on such neighborhoods. Thus, a sequence of real-analytic functions $\phi_j$ converges to a real-analytic function $\phi$ on $V$ in $C^\omega$ iff there exists a neighborhood $V_\rho$ of 
$V$ in the complexification $\C^{2n}/\Z^{2n}$ where $\phi$ and every $\phi_j$ for $j$ large enough have their analytic extensions well defined and $\sup_{V_\rho} \|\phi_j-\phi\|\to 0$ as $j\to+\infty$.  The space   
$\Ham^\omega (V)$  is formed by analytic mappings; it is considered with the inductive limit topology as described above. 

 Both $\Ham^\infty (V)$ and
$\Ham^\omega (V)$ are groups (this follows from the identity $\phi^{0,t}_{G}\circ \phi^{0,t}_{H}=\phi^{0,t}_K$
where $K_t=G_t + H_t \circ (\phi^{0,t}_{H_2})^{-1}$).

The simplest examples of Hamiltonian maps are given by {\em vertical} and {\em horizontal shear maps}: 
\begin{itemize}
\item a horizontal shear $(q,p)\mapsto (q + \nabla \tau(p),p)$ is the time-1 map for the time-independent Hamiltonian 
$H(q,p)=\tau(p)$, where $\tau\in C^\omega(\T^n, \R)$;
\item a vertical shear $(q,p)\mapsto (q, p - \nabla v(q))$ is the time-1 map for the time-independent Hamiltonian 
$H(q,p)=v(q)$, where $v\in C^\omega(\T^n, \R)$.
\end{itemize}
The system of differential equations defined by Hamiltonian $H(q,p)=\tau(p)$ is
$$\dot q =\nabla \tau(p),\quad \dot p=0;$$
its flow map $\phi^t_H: (q,p)\mapsto (q + t \nabla \tau(p),p)$ is a horizontal shear for every $t\in \R$. Similarly, the flow map 
$\phi^t_H: (q,p)\mapsto (q, p - t \nabla v(q))$ for $H(q,p)=v(q)$ is a verticall shear for every $t\in \R$. We see that
the vertical and horizontal shear maps form Abelian subgroups of $\Ham^\omega (V)$, which we denote as
$\cal V$ and, respectively, $\cal T$.
 \begin{theorem*}(Main) The group $<\cal V , \cal T>$ generated by $\cal V$ and $\cal T$  is dense in $\Ham^\omega (V)$. In other words, every real-analytic Hamiltonian diffeomorphism of $V$ can be $C^\omega$-approximated by a composition 
$S_M\circ \ldots \circ S_1$ where $S_j\in \cal T\cup \cal V$, $j=1,\ldots,M$. 
\end{theorem*}

The proof is given in the next Section. Since $\Ham^\omega (V)$ is $C^\infty$-dense in $\Ham^\infty (V)$, we obtain
\begin{corollary}The group  generated by $\cal V$ and $\cal T$ is dense in $\Ham^\infty (V)$.\end{corollary}
Note that any real-analytic function on $\R^n$  can be arbitrarily well approximated, on any given compact, by a periodic function with a sufficiently large period. Therefore, the lifts of Hamiltonian maps of $\T^n\times \T^n$ approximate (on any given compact)
Hamiltonian maps of an annulus $\T^n\times \R^n$ or a ball $\R^n\times \R^n$. This implies
\begin{corollary}
The theorem extends to the cases where $V=\T^n\times \R^n$ and $V=\R^n \times \R^n$: every Hamiltonian
map is approximated by a composition of vertical and horizontal shears.\end{corollary} 

Along the proof of the main theorem, we will check that the ``parametric version'' of the results also holds. Namely,
we have the following 
\begin{corollary}\label{mainpara}
Given a compact real-analytic manifold $\cP$, every analytic family $f_\cP=(f_\mu)_{\mu\in \cP}$ of Hamiltonian diffeomorphisms of $\T^n\times \T^n$, $\T^n\times \R^n$ or $\R^{n}\times \R^n$ can be arbitrarily well approximated by analytic families of compositions of vertical and horizontal shears. For any compact set $\cP$, every continuous family $f_\cP=(f_\mu)_{\mu\in \cP}$ of analytic Hamiltonian diffeomorphisms of $\T^n\times \T^n$, $\T^n\times \R^n$ or $\R^{n}\times \R^n$ can be arbitrarily well approximated by continuous families of compositions of vertical and horizontal shears.
\end{corollary}
To be precise, we recall that a sequence of continuous families  $(f_{j,\mu})_{\mu\in \cP}$ of analytic diffeomorphisms on $V$ converges to $(f_\mu)_{\mu\in \cP}$ if there exists a complex neighborhood $V_\rho$ of $V$ such that for every $\delta>0$, for every $j$ large enough, 
$\sup_{(\mu,x)\in\cP\times V_\rho} \|f_\mu(x)-f_{j,\mu}(x)\|<\delta$. Also, we call a family  
$(f_\mu)_{\mu\in \cP}$ analytic if $f_\mu(x)$ is a real-analytic function of $\mu$ and $x$; the convergence in Corollary \ref{mainpara} is then in $C^\omega(\cP\times V, V)$. These two settings (of continuous and analytic families) seem to be most natural. In order to consider them in a unified way, we adopt from now on a more general setting where 
$\cP$ denotes allways a product:
\[\cP=\cP_1\times \cP_2\]
of a compact set $\cP_1$ and a compact analytic manifold $\cP_2$. Also from now on, the considered family $f_\cP= (f_\mu)_{\mu\in\cP:=\cP_1\times\cP_2}$ will be  
 analytic in $\mu_2$ and continuous with respect to $\mu_1$. Each map $f_\mu(x)$ is the time-1 map of a time-dependent Hamiltonian $H_{\mu,t}(x)$; we assume that $H$ is an analytic function of $(\mu_2,x)$ and a continuous (in the topology of $C^\omega(\cP_2\times V,\R)$) function of $(\mu_1,t)$, and say that the family $f_\cP$ is {\em generated} by the family $(H_{\mu,t})_{\mu,t}$.

\begin{remark}
For the ease of presentation, the main theorem is given for $\T^n\times \T^n$ endowed with the standard symplectic form $\omega= \sum_i dp_i\wedge dq_i$, but the proof, with obvious modifications, works also for any symplectic form of the form $\sum_i a_i \ dp_i\wedge dq_i$, with {\em constant} $a_i>0$. A natural question is how to extend the results to other symplectic forms on the torus or to other product symplectic manifolds. 
\end{remark}

The main Theorem implies\footnote{A shear map of  $\R^{2n}$ is the composition of two H\'enon maps. So the main result implies that compositions of H\'enon maps form a dense set in $\Ham^\infty(\R^{2n})$.} 
the work \cite{TUR02} where symplectic maps of $V=\R^{2n}$ were considered in the smooth case. 
 While we found a way to extend the method of \cite{TUR02} to the annulus case, we prefer to present here a 
more powerful approach, inspired by a technology developed in \cite[\textsection 2.4-3.2 and app. A]{BGH22} for the non-symplectic case. Similar results for holomorphic automorphisms of $\C^n$, including the volume-preserving case, were obtained in \cite{A,AL} and have played an important role in solving several problems of complex analysis, see review in \cite{FK21}. The symplectic result of \cite{TUR02} for $V=\R^{2n}$ was key for the proof of
the genericity of the ``ultimately rich'' (universal) dynamics for certain classes of symplectic and non-symplectic maps \cite{GST07,GT10,Tu15,GT17}
and for the proof of Herman's metric entropy conjecture \cite{BT19}.
It has also been used in algorithms for physics-informed machine learning \cite{Jin,Burby,valperga2022learning}. The current result and its short constructive proof  for the annulus 
$V=\T\times \R$ enabled to disprove the Birkhoff conjecture in \cite{Be22pseudo}.  

\section{Proof of the main theorem}
We  use  the Poisson algebra structure on $C^\omega(V,\R)$, which is the Hamiltonian counterpart of the Lie algebra structure on the space of vector fields on $V$. Namely, given two functions $f,g\in C^\omega(V, \R)$ the \emph{Poisson bracket} $\{f,g\}$ is the function defined by
\[\{f,g\}= \sum_i \partial_{q_i} f \cdot  \partial_{p_i} g- \partial_{q_i} g \cdot  \partial_{p_i} f\; .\]
It is easy to check that the Lie bracket of the Hamiltonian vector fields $X_{f}$ and $X_{g}$ is the Hamiltonian vector field $X_{\{f,g\}}$.

Cartan's Theorem establishes a correspondence between closed subgroups and Lie sub-algebras for 
finite-dimensional Lie groups. Certain aspects of this correspondence have been generalized in \cite[Prop. B.1]{BGH22} for the group of compactly supported smooth diffeomorphisms. Below is the counterpart for the group of analytic Hamiltonian diffeomorphisms: 
\begin{proposition}\label{poisson lemma0}Let a set $G$ be a closed subgroup of $\Ham^\omega(V)$. Let 
$\cal P(G)$ be the set of all time-independent Hamiltonians
$H\in C^\omega(V,\R)$ such that
their flow maps $\phi^t_H$ belong to $G$ for all $t\in\R$:
\[\cal P(G):= \{H\in C^\omega(V,\R): \phi^t_H\in G,\; \forall t\}\; .\]
Then $\cal P(G)$ is a closed Lie sub-algebra of $C^\omega(V)$. In other words, it is a closed vector subspace of $C^\omega(V)$ and the Poisson bracket of any two  functions from $\cal P(G)$ also belongs to $\cal P(G)$.
\end{proposition}
\begin{proof} First note that $\cal P(G)$ is closed by continuity of $H\mapsto  \phi^t_H$ for every $t\in \R$. Then the proposition follows from the two lemmas below. \end{proof}
\begin{lemma}\label{claim1}The set  $\cal P(G)$ is a vector space.\end{lemma}
\begin{lemma}\label{claim2}The vector space   $\cal P(G)$ is a Lie algebra.\end{lemma}
\begin{proof}[Proof of \cref{claim1}]Let  $H\in \cal P(G)$ and $\lambda\in \R$. For every $t\in \R$, the map  
$\phi^{\lambda  t} _{H}=\phi^{t} _{\lambda H}$ belongs to 
$G$ and so $\lambda H \in \cal P(G)$. Hence, it suffices to  show that $H_1+H_2\in \cal P(G)$, for every   $H_1,H_2\in \cal P(G)$.
As $\phi^t_{H_1+H_2}= \phi^1_{tH_1+tH_2}$, and $tH_1, tH_2\in \cal P(G)$ whenever $H_1,H_2\in \cal P(G)$ as we just showed, it suffices to check that 
$\phi^1_{H_1+H_2}\in G$ for every $H_1,H_2\in \cal P(G)$. 

On a complex extension $V_0$  of $V$, the following holds uniformly as $t\to 0$:
 \begin{multline*}\phi^{t}_{H_1}\circ \phi^{t}_{H_2}=id + t\cdot \partial_t (\phi^{t}_{H_1}\circ \phi^{t}_{H_2})_{t=0} +O(t^2)=id + t\cdot (X_{H_1}+X_{H_2})+O(t^2)\\
= id + t\cdot \partial_t (\phi^{t}_{H_1+H_2})_{t=0} +O(t^2)=\phi^{t}_{H_1+H_2}\; .\end{multline*}
In particular, there exists $C>0$   such that for every $N\ge 1$
\begin{equation}\label{def of C}\sup_{ V_{0}} \|\phi^{1/N}_{H_1+H_2}-\phi^{1/N}_{H_1}\circ \phi^{1/N}_{H_2} \|\le C \cdot N^{-2} \; .\end{equation}
Taking sufficiently small complex neighborhoods $V_2\Subset V_1\subset V_0$ of $V$ and applying Discretization Lemma \ref{lemma a ecrie}, we infer from (\ref{def of C}) that  
  \[\sup_{V_2}\|\phi_{H_1+H_2}^{1}-(\phi^{1/N}_{H_1}\circ \phi^{1/N}_{H_2})^N \|\le   C\exp(L)  N^{-1} \quad \forall N\ge N_0\]
for some constants $L$ and $N_0$ (in the Discretization Lemma, put $\varepsilon_N:=  C/N$, $\phi^{s,t}:= \phi_{H_1+H_2}^{t-s}$, and $g_j=\phi^{1/N}_{H_1}\circ \phi^{1/N}_{H_2}$ for all $j$). 
  
Thus $\phi_{H_1+H_2}^{1}$ is arbitrarily close to the  element $(\phi^{1/N}_{H_1}\circ \phi^{1/N}_{H_2})^N$ of the group  $G$. As $G$ is closed, it follows that  $ \phi^{1}_{H_1+H_2}\in G$,  as required.  \end{proof}

\begin{proof}[Proof of \cref{claim2}] It suffices to show that for any $H_1,H_2\in \cal P(G)$, the function $H_3=\{H_1,H_2\}$ belongs 
to $\cal P(G)$. Since $\phi^t_{H_3}=\phi^1_{\{tH_1,H_2\}}$ and $tH_1\in \cal P(G)$ whenever 
$H_1\in \cal P(G)$, we only need to show that  $\phi^1_{H_3}$ belongs to $G$.
On a complex extension $V_0$  of $V$, we have, uniformly as  $t\to 0$:
\begin{equation} \label{taylor} \phi^{t}_{H_j}=   id     +t X_{H_j}  + \frac{t^2}{2} D X_{H_j} \cdot X_{H_j}+O(t^3)   \;,  \; j=1,2,3.\end{equation}
So,
\[ 
\phi^{t/N}_{H_1}\circ \phi^{t/N}_{H_2}=  id + \tfrac{t}{N} (X_{H_1}+X_{H_2})+\tfrac{t^2}{2N^2} (D X_{H_1} \cdot X_{H_1} 
+D X_{H_2} \cdot X_{H_2}+2 D X_{H_1} \cdot X_{H_2})
+ O(t^3) \]
\[ 
\phi^{-t/N}_{H_2}\circ \phi^{-t/N}_{H_1}= id - \tfrac{t}{N} (X_{H_1}+X_{H_2})+\tfrac{t^2}{2N^2} (D X_{H_1} \cdot X_{H_1} 
- D X_{H_2} \cdot X_{H_2}+2 D X_{H_2} \cdot X_{H_1}) +O(t^3).  \]
Thus,  uniformly on $V_0$ as $t\to 0$,
\begin{equation}\label{prelie} \phi^{-t }_{H_1}\circ \phi^{-t }_{H_2}\circ \phi^{t }_{H_1}\circ \phi^{t }_{H_2}=  id +
 {t^2}  (D X_{H_1} \cdot X_{H_2} - D X_{H_2} \cdot X_{H_1}) +O(t^3).\end{equation} 
One can check that $D X_{H_1} (X_{H_2}) - D X_{H_2} (X_{H_1})$ is the symplectic gradient of $H_3=\{H_1,H_2\}$.  
Thus  \cref{taylor} at $j=3$ and \cref{prelie} imply the existence of $C>0$ such that for $N\equiv t^{-2}$ sufficiently large,
\begin{equation}\label{prelie2}\sup_{V_{0}} \left\| \phi^{1/N }_{H_3}-\phi^{-1/\sqrt N}_{H_1}\circ \phi^{-1/\sqrt N}_{H_2}\circ \phi^{1/\sqrt N}_{H_1}\circ \phi^{1/\sqrt N}_{H_2}
 \right\| \le C  \frac{1}{N^{3/2}}  .\end{equation}
 From this, taking sufficiently small complex neighborhoods $V_2\Subset V_1\subset V_0$, one finds that the assumptions of the Discretization Lemma \ref{lemma a ecrie} are satisfied with $\varepsilon_N:=  C/\sqrt{N}$, $\phi^{s,t}:= \phi_{H_3}^{t-s}$, and $g_j=\phi^{-1/\sqrt N}_{H_1}\circ \phi^{-1/\sqrt N}_{H_2}\circ \phi^{1/\sqrt N}_{H_1}\circ \phi^{1/\sqrt N}_{H_2}$ for all $j$. This implies that for some $L$ and $N_0$  
  \[\sup_{V_2}\|\phi_{H_3}^{1}-(\phi^{-1/\sqrt N}_{H_1}\circ \phi^{-1/\sqrt N}_{H_2}\circ \phi^{1/\sqrt N}_{H_1}\circ \phi^{1/\sqrt N}_{H_2})^N \|\le  C\exp(L)  N^{-1/2} \; , \quad \forall N\ge N_0\; .\]
Thus $\phi_{H_3}^{1}$ is arbitrarily close to the  element $(\phi^{-1/\sqrt N}_{H_1}\circ \phi^{-1/\sqrt N}_{H_2}\circ \phi^{1/\sqrt N}_{H_1}\circ \phi^{1/\sqrt N}_{H_2})^N$ of the group  $G$. As $G$ is closed, it follows that  $ \phi^{1}_{H_3}\in G$,  as required.  
\end{proof}
The space of families of analytic Hamiltonian maps endowed with the composition rule $(f_\mu)_{\mu\in \cP}\circ (g_\mu)_{\mu\in \cP}= (f_\mu\circ g_\mu)_{\mu\in \cP}$ is a group. One can check that the proof of the above proposition does not alter as long as $\cP$ is compact, by using the parametric counterpart \cref{lemma a ecriepara} of \cref{lemma a ecrie}. Namely, we have
 \begin{corollary}\label{remark pour la prop indepara} If $G_\cP$ is a closed subgroup of the space of
families of analytic Hamiltonian maps 
then the following is a closed sub-algebra of the space of families of functions on $V$: 
 \[\cal P(G_\cP):= \{(H_\mu)_{\mu\in \cP}  : (\phi^t_{H_\mu})_{\mu\in \cP} \in G_\cP,\; \forall t\}\; .\]\end{corollary}  

We apply \cref{poisson lemma0} to the group $G$ obtained by taking the $C^\omega$-closure of the group $<\cal V , \cal T>$ generated by $\cal V$ and $\cal T$:
\[G= cl(<{\cal V} , {\cal T}>)\; .  \]
Recall that the groups $\cal V$ and $\cal T$ of vertical and horizontal shears consist of the time-1 maps for the time-independent Hamiltonian functions which depend only on $p$ or, respectively, only on $q$. For such functions, the time-$t$ map belongs to $\cal V$ or, respectively, $\cal T$
for all $t\in\R$. Thus, the $C^\omega$ Hamiltonians of the form $H(q,p)=\tau(p)$ or $H(q,p)=v(q)$ belong to the Lie algebra 
$\cal P(G)$. 
This implies, by Proposition \ref{poisson lemma0}, that every Hamiltonian of the form
\begin{equation}\label{defcalp}
H(q,p)=v_0(q) + \sum_{1\le r\le R} \{w_r(q),\{v_r(q),\tau_r(p)\}\}
\end{equation}
lies in $\cal P(G)$; here $R\ge 0$ and $\tau_r,v_r,w_r\in C^\omega(\T^n,\R)$. 

We denote the set comprised by $C^\omega$-functions $\T^n\times\T^n\to \R$ which can be represented in the form \eqref{defcalp} 
as $\mathring{\cal P}$. In short one can denote:
\[\mathring {\cal P}= \cal V+\{\cal V,\{ \cal V, \cal T\}\}\; .\]

  As we said, $\mathring{\cal P}\subset \cal P(G)$, i.e., for every $H\in \mathring{\cal P}$ its flow 
maps $\phi^t_H$
can be arbitrarily well approximated by compositions of vertical and horizontal shears.

\begin{proposition}\label{poissonclaim0}The set  $\mathring{\cal P}$ is dense in $C^\omega(V,\R)$. 
\end{proposition} 
\begin{proof} By  Fourier's Theorem, any $C^\omega$ function on the torus $V$ can be approximated by a trigonometric polynomial, i.e., a function 
of the form
$$\sum_{\max_j\{|m_j|,|k_j|\}\le  K}  \Re(c_{m,k}  e^{2\pi i (<m,q>+<k,p>)}),$$ where $m$ and $k$ are integer-valued 
$n$-vectors, and $<\cdot,\cdot>$ denotes the inner product.
 Therefore, it is enough to show that every trigonometric polynomial
belongs to $\mathring{\cal P}$, i.e., it has the form (\ref{defcalp}) for some choice of the functions $v$, $w$, and $\tau$. Thus, we choose
$v_0(q) = \sum_{\max_j |m_j|\le  K}  \Re(c_{m,0} e^{2\pi i <m,q>})$, and it remains to show that for every $k$ and $m$ such that $k\neq 0$ 
the term $\Re(c_{m,k}  e^{2\pi i (<m,q>+<k,p>)})$ is a linear combination of terms which can be represented as $\{w_r(q),\{v_r(q),\tau_r(p)\}$.

This is done as follows. Since $k\neq 0$, there exists an index $j$ such that $k_j\neq 0$. 
Hence there is $\sigma\in \{-1,1\}$ such that $A= (\sum_{s=1}^n  {m_s k_s} ) - \sigma  k_j $ is not zero.

We denote $m'=(m'_1,\ldots, m'_n)$,
where $m_j'=m_j-\sigma$ and $m_s'=m_s$ if $s\neq j$. 
Then 
$$e^{2\pi i (<m,q>+<k,p>)}=e^{2\pi i <m',q>} e^{2\pi i <k,p>} e^{2\pi i \sigma q_j},$$
so $\Re(c_{m,k}  e^{2\pi i (<m,q>+<k,p>)})$ is a linear combination of products of cosines or sines of 
$2\pi <m',q>$, $2\pi <k,p>$, and $2\pi \sigma q_j$. Hence, it suffices to show that for every
$(\alpha,\beta,\gamma)$ there exist functions $w$, $v$, and $\tau$ such that
\begin{equation}\label{sinsin}
\sin(2\pi <m',q> +\alpha)\; \sin(2\pi <k,p> + \beta)\; \sin(2\pi \sigma q_j + \gamma)=
\{w(q),\{v(q),\tau(p)\}.
\end{equation}
For that, we choose 
\[ v(q) = \frac{\cos(2\pi <m',q> + \alpha)}{2\pi A} ,\qquad  \tau(p)=-\frac{\sin(2\pi <k,p> +\beta)}{2\pi}, \qquad
w(q)= \frac{\cos(2\pi \sigma q_j + \gamma)}{4\pi^2 k_j \sigma}.\]
Now, we have: 
\[\{v(q),\tau(p)\}= <\nabla_q v , \nabla_p \tau >= \sum_{s=1}^n \frac{m_s' k_s}A \sin(2\pi <m',q> +\alpha)\cos(2\pi <k,p> +\beta)\]
and as $A= \sum_{s=1}^n  m_s' k_s $,
\[\{v(q),\tau(p)\}=\sin(2\pi <m',q> +\alpha)\cos(2\pi <k,p> +\beta) .\]
As  $w$ is a function of only one variable, $q_j$, we have:
\[\{w(q),\{v(q),\tau(p)\}\}= w'(q_j) \partial_{p_j}\{v(q),\tau(p)\}  ,\]
which gives \cref{sinsin} since $\partial_{p_j}\{v(q),\tau(p)\} = -2\pi \cdot k_j\cdot \sin(2\pi <m',q> +\alpha)\sin(2\pi <k,p> +\beta)$ and $w'(q_j) = -\frac{\sin(2\pi \sigma q_j + \gamma)}{2\pi  k_j }$.  
\end{proof} 

Proposition \ref{poissonclaim0} implies that the flow maps $\phi^t_H$ of any time-independent Hamiltonian 
$H\in C^\omega(V,\R)$ can be arbitrarily well approximated by the flow maps of some Hamiltonians from 
$\mathring{\cal P}$.
Since $\mathring{\cal P}\subset \cal P(G)$ and $\cal P(G)$ is closed subset of $C^\omega(V,\R)$, this gives $\cal P(G)=C^\omega(V,\R)$ and so:
\begin{proposition}\label{lapropinde}For every time-independent Hamiltonian $H$ its time-$t$ maps 
can, for every~$t$, be arbitrarily well approximated by compositions of vertical and horizontal shears. \end{proposition}
Using \cref{remark pour la prop indepara} and the fact that the approximation given by the proof of  \cref{poissonclaim0} is based on the Fourier's decomposition which depends analytically on the function considered, we obtain
\begin{corollary}\label{lapropindepara} Every family of  time-independent Hamiltonian functions $(H_\mu)_{\mu\in \cP}$, the family of its time-$t$ maps 
can, for every~$t$, be arbitrarily well approximated by families of compositions of vertical and horizontal shears.\end{corollary}

To finish the proof of the theorem, we now show that the same is true for time-dependent Hamiltonians. First,
we prove
\begin{lemma}\label{dernierlemma}
Every diffeomorphism $f \in \Ham^\omega (V)$ can be approximated by a composition of flow maps defined by time-independent Hamiltonians.
More precisely, if $f$ is the flow map $\phi^{0,1}_H$ of a time-dependent Hamiltonian $H=(H_t)_{t\geq 0}$, then there is a complex neighborhood $W$ of $V$ such that for every $\eta>0$, the holomorphic extension
$f|_W$ is $\eta$-close, for all sufficiently large $N$, to the composition 
of time-$1/N$ maps $\phi^{1/N}_{H_{j/N}}$ for the time-independent Hamiltonians $H_{j/N}$:
\[\sup_{W} \left \|f- \phi^{1/N}_{H_{(N-1)/N}}\circ \ldots \circ \phi^{1/N}_{H_{1/N}}\circ \phi^{1/N}_{H_{0}} \right\| < \eta\; .\]
\end{lemma}
 \begin{proof}
We have:
\[f=\phi^{0,1}_H= \phi^{(N-1)/N, 1}_H\circ \phi^{(N-2)/N, (N-1)/N}_H \circ \ldots\circ \phi^{1/N, 2/N}_H \circ \phi^{0,1/N}_H.\]
Using the compactness of the time interval $[0,1]$, there exists a complex neighborhood $V_0$ of $V$ and a sequence 
$\varepsilon_N\to 0$ such that the flow maps for time-independent Hamiltonians satisfy
\[ \sup_{V_{0}} \left\| \phi^{1/N}_{H_t} - id -  \frac{1}{N} X_{H_t}\right\|\le     \frac{\varepsilon_N}{2N}, \quad \forall 0\le t\le 1, \forall N\ge 1,\]
and  the non-autonomous flow maps satisfy
\[ \sup_{V_{0}} \left\| \phi^{t, t+1/N}_H- id -  \frac{1}{N} X_{H,t}\right\|=
\sup_{V_{0}} \left\| \int_{t}^{t+1/N}  X_{H,s} \circ \phi^{t,s}_H ds -   \frac{1}{N} X_{H,t}\right\|
\le   \frac{\varepsilon_N}{2N}, \quad \forall 0\le t\le 1-1/N.\]
By \cref{nona}, the vector fields $X_{H,t}$ and $X_{H_t}$ are equal. Thus, summing these two estimates taken at $t=j/N$, we obtain:
\[ \sup_{V_{0}} \left\| \phi^{j/N, (j+1)/N}_H- \phi^{1/N}_{H_{j/N}}  \right\|\le   \frac{\varepsilon_N}{N}, \quad \forall 0\le j\le 1-1/N\; .\]
From this, it easy to find complex neighborhoods $W=V_2\Subset V_1\subset V_0$  which 
 satisfy the assumptions of \cref{lemma a ecrie} with $\phi^{s,t}:= \phi_H^{s,t}$ and $g_j:= \phi^{1/N}_{H_{j/N}}$. This implies the sought bound: there exists $L>0$ such that for all $N$ large enough 
\[ \sup_{V_2} \left\|  \phi^{0,1}_H- g_{N-1} \circ \ldots \circ g_0
\right\|  \le   \exp(L)\cdot  \varepsilon_N\; . \]
\end{proof}
If, in the above proof, we consider parametric families and employ \cref{lemma a ecriepara2} instead of \cref{lemma a ecrie}, we obtain
\begin{corollary}\label{dernierlemmapara} Every family of analytic Hamiltonian diffeomorphisms
$(f_\mu)_{\mu\in \cP}$  can be approximated by compositions of families of flow maps
defined by time-independent Hamiltonians taken from the family of Hamiltonians that generates $(f_\mu)_{\mu\in \cP}$.\end{corollary}  

Now, let $f$ be the time-1 map for a time-dependent Hamiltonian $(H_t)_{0\le t\le 1}$. For every frozen
value of $s\in[0,1]$, consider the function $H_s: V\to \R$ as a time-independent Hamiltonian and take its time-$t$ flow map. This defines a family of maps  
\begin{equation}\label{famphiths}
(\phi^t_{H_s})_{0\le s,t\le 1}.
\end{equation}
It is a continuous $2$-parameter family of time-independent Hamiltonian flow maps; hence by \cref{lapropindepara}, there exists a complex neighborhood $W_0$ of $V$  such that for every 
$\eta>0$ there is a continuous family $(g_s^t)_{0\le s,t\le 1}$ of compositions of vertical and horizontal shears, such that:
\[ \sup_{W_{0}} \|\phi^t_{H_s}- g_s^t\| \le \eta \; ,\quad \forall 0\le s,t\le 1\; .\]
Thus, for any complex neighborhood $W_1\Subset W_0$, if $\eta$ is small enough, then the differences
 \[ (\phi^{1/N}_{H_{(N-1)/N}}\circ \ldots \circ \phi^{1/N}_{H_{j/N}} -\phi^{1/N}_{H_{(N-1)/N}}\circ \ldots \circ \phi^{1/N}_{H_{(j+1)/N}}
 \circ g_{j/N }^{1/N})\circ g_{(j-1)/N}^{1/N}\circ \ldots g_0^{1/N} \]
are uniformly small (for all $N$ and all $j=1,\ldots, N-1$) on $W_{1}$.

Summing this over $1\le j\le N-1$, we obtain that $ \phi^{1/N}_{H_{(N-1)/N}}\circ \ldots \circ \phi^{1/N}_{H_{1/N}}\circ \phi^{1/N}_{H_{0}} $ and $ g_{(N-1)/N}^{1/N}\circ \ldots g_0^{1/N}$ are uniformly close on $W_{1}$. Now, we invoke \cref{dernierlemma}, which gives the existence of a complex neighborhood $W_2\subset W_1$ such that for every $\delta>0$, if $N$ is sufficiently large, then 
  $f$ is $\delta/2$ -close to   $ \phi^{1/N}_{H_{(N-1)/N}}\circ \ldots \circ \phi^{1/N}_{H_{1/N}}\circ \phi^{1/N}_{H_{0}} $  on $W_2$. Thus, for $N$ large enough, we have
\[\sup_{W_{2}} \left \|  f- g_{(N-1)/N}^{1/N}\circ \ldots g_0^{1/N} \right\| \le \delta.\]
Since $W_2$ does not depend on $\delta$, it follows that $ \phi^{0,1}_H$  can be approximated arbitrarily well by
a composition of vertical and horizontal shears. This proves the theorem.
\cref{mainpara} is proved exactly in the same way, just the family (\ref{famphiths}) now also depends on the additional parameters $\mu\in\cP$ and, instead of \cref{dernierlemma} we employ its prametric version \cref{lapropindepara}.
 
\section{Bounds on compositions}
The following lemma was used several times in the proof above. It works actually on any analytic manifold and the dynamics does not need to be Hamiltonian nor real. 
Let $V_0$ be a complex manifold and   $V_1$  a neighborhood  of a compact subset $V_2$ of $V_0$: $ V_2\Subset V_1\subset V_0$. 
\begin{lemma}\label{lemma a ecrie}{\rm \bf (Discretization lemma)}
For any $L>0$ and any positive sequence $\varepsilon_N\to 0$, there is  $N_0$ such that the following property holds true for every $N\ge N_0$.
\begin{enumerate}[$(i)$]
\item Let $X:=(X_t)_{t\in[0,1]}$ be any time-dependent  vector field, holomorphic on $V_0$ and continuously dependent on time, such that  
its  flow $(\phi^{s,t})_{0\le s\le t\le 1}$ is defined on $V_1$ and its  derivative is bounded by $L$:
\[V^{t}_{1}:= \phi^{0,t}(V_1)\subset V_0\qand \sup_{x\in V_1^t} \|\partial_x X_t\| \leq L\quad \forall t\in [0,1].\]
\item  
Let 
 any sequence of analytic maps $g_{j}$, $0\le j \le N-1$, be defined on
$V^{j/N}_{1}$ and satisfying
\begin{equation}
\label{hyponorm}\sup_{V^{j/N}_1} \left\| \phi^{j/N, (j+1)/N} - g_{j}\right\| < \frac{\varepsilon_N}N \; , \quad \forall 0\le j< N\; .\end{equation}
\end{enumerate}
Then the composition  $g_{N-1} \circ \ldots \circ g_{0}$ is well-defined on $V_{2}$ and satisfies
\begin{equation}\label{disclem}
\sup_{V_2 } \left\| \phi^{0, 1} - g_{N-1} \circ \ldots \circ g_{0}\right\| < \exp(L) \cdot \varepsilon_N .
\end{equation}
\end{lemma} 
\begin{proof} We show, by induction in $k$, that for every $k=1,\ldots, N$ the composition
$g_{k-1} \circ \ldots \circ g_{0}$ is well-defined on $V_{2}$ and satisfies
\begin{equation}\label{indkL}
\sup_{V_2 } \left\| \phi^{0, k/N} - g_{k-1} \circ \ldots \circ g_{0}\right\| < \frac{1+\ldots +\exp(kL/N)}{N} \cdot \varepsilon_N
\end{equation}
for all $N\geq N_0$. Obviously, this gives (\ref{disclem}) at $k=N$.

Note that by Gr\"onwall's inequality 
\begin{equation}\label{gron}
\sup_{x\in V_{1}^{s}}\|\partial_x \phi^{s,t}\|\leq \exp(L(t-s)), \qquad 0\leq s\leq t\leq 1.
\end{equation}
In particular, the maps $\phi^{s,t}$ are uniformly continuous, hence
there is $\eta>0$ (depending only on $L$) such that  $V^{t}_{1}$ contains the $\eta$-neighborhood of 
$V^{t}_{2}:=\phi^{0,t}(V_2)$ for every $0\le  t\le 1$. 

Now, assume (\ref{indkL}) is true for some $k\leq N$ (it is true at $k=1$ by assumption). This implies
$$\sup_{V_2 } \left\| \phi^{0, k/N} - g_{k-1} \circ \ldots \circ g_{0}\right\| < \exp(L) \varepsilon_N,$$
hence for all sufficiently large $N$:
$$\sup_{V_2 } \left\| \phi^{0, k/N} - g_{k-1} \circ \ldots \circ g_{0}\right\|  <\eta$$
Therefore, the image $g_{k-1} \circ \ldots \circ g_{0}(V_2)$ lies in the $\eta$-neighborhood of 
$V^{k/N}_{2}$. This is a subset of $V^{k/N}_{1}$ where $g_k$ is defined by assumption, so the composition 
$g_{k} \circ \ldots \circ g_{0}$ is well-defined on $V_2$, as required. Since
$$\phi^{0,(k+1)/N}=\phi^{k/N,(k+1)/N} \circ \phi^{0,k/N},$$
it follows from (\ref{gron}) and (\ref{indkL}) that
$$
\sup_{V_2 } \left\| \phi^{0, (k+1)/N} - \phi^{k/N,(k+1)/N} \circ g_{k-1} \circ \ldots \circ g_{0}\right\| < \exp(L/N)
\frac{1+\ldots +\exp(kL/N)}{N} \cdot \varepsilon_N.
$$
By (\ref{hyponorm}) at $j=k$, we have
$$
\sup_{V_2 } \left\| (\phi^{k/N,(k+1)/N} - g_k)\circ g_{k-1} \circ \ldots \circ g_{0} \right\| < 
\frac{1}{N} \cdot \varepsilon_N.
$$
Summing up these two inequalities, we obtain inequality (\cref{indkL}) at $k+1$, i.e., we complete the
induction step.
\end{proof}
Let us emphasis that the bound on the above lemma depends only on $L$ and $ V_2\Subset V_1\subset V_0$. So it implies immediately the following  for  family parametrized by a set $E$ (not necessarily topological).
\begin{corollary}\label{lemma a ecriepara}
For any $L>0$ and any positive sequence $\varepsilon_N\to 0$, there is  $N_0$ such that the following property holds true for any $N\ge N_0$.

Let $(X_\mu)_{\mu\in E}$ be  any families  of time dependent vector fields $X_\mu$ and let $(g_{j,\mu})_{\mu\in E}$, $0\le j \le N-1$, be any families of  maps such that $X_\mu$ and $(g_{j,\mu})_{0\le i\le N-1}$ satisfy assumptions $(i)$ and $(ii)$ of \cref{lemma a ecrie} for every $\mu\in E$.
Then each of the composition  $g_{N-1,\mu} \circ \ldots \circ g_{0,\mu}$ is well defined on $V_{2}$ and satisfies the following estimate with the flow $\phi^{0, 1}_\mu $ of $X_\mu$:
\[\sup_{V_{2}} \left\| \phi^{0, 1}_\mu - g_{N-1,\mu} \circ \ldots \circ g_{0,\mu}\right\| < \exp(L) \cdot \varepsilon_N.\]
\end{corollary}
Now assume that $V_2\subset V_1\subset V_0$ are complex extension of a compact real analytic manifold $V$ and that $E$ is of the form  $E= \cP_1\times \tilde \cP_2$ where $\cP_1$ is a compact set and $\tilde \cP_2$ a complex extension of an analytic compact manifold $\cP_2$. We obtain immediately
\begin{corollary}\label{lemma a ecriepara2}
For any $L>0$ and any positive sequence $\varepsilon_N\to 0$, there is  $N_0$ such that the following property holds true for any $N\ge N_0$.
Let $(X_\mu)_{\mu\in \cP}$ be  any families  of time dependent $C^\omega$-vector fields on $V$ and let $(g_{j,\mu})_{\mu\in \cP}$, $0\le j \le N-1$, be any continuous families of  $C^\omega$ maps of $V$ which all extend to $ \cP_1\times \tilde \cP_2$ and such that $X_\mu$ and $(g_{j,\mu})_{0\le i\le N-1}$ satisfy assumptions $(i)$ and $(ii)$ of \cref{lemma a ecrie} for every $\mu\in \cP_1\times \tilde \cP_2$.
Then each of the composition  $g_{N-1,\mu} \circ \ldots \circ g_{0,\mu}$ is well defined on $V_{2}$ and satisfies the following estimate with the flow $\phi^{0, 1}_\mu $ of $X_\mu$:
\[\sup_{V_{2}} \left\| \phi^{0, 1}_\mu - g_{N-1,\mu} \circ \ldots \circ g_{0,\mu}\right\| < \exp(L) \cdot \varepsilon_N.\]
\end{corollary}
%
%

\bibliographystyle{alpha}
\bibliography{references.bib}

\newcommand{\etalchar}[1]{$^{#1}$}
\def\cprime{$'$} \def\cprime{$'$} \def\cprime{$'$}
\begin{thebibliography}{VWT{\etalchar{+}}22}

\bibitem[AL92]{AL}
E.~Anders{\'e}n and L.~Lempert.
\newblock On the group of holomorphic automorphisms of $\mathbb { C}^n$.
\newblock {\em Inventiones mathematicae}, 110(1):371--388, 1992.

\bibitem[And90]{A}
E.~Anders{\'e}n.
\newblock Volume-preserving automorphisms of $\mathbb { C}^n$.
\newblock {\em Complex Variables, Theory and Application: An International
  Journal}, 14(1-4):223--235, 1990.

\bibitem[Ber22]{Be22pseudo}
P.~Berger.
\newblock Analytic pseudo rotation.
\newblock {\em ArXiv}, 2022.

\bibitem[BGH22]{BGH22}
P.~Berger, N.~Gourmelon, and M.~Helfter.
\newblock Any diffeomorphism is a total renormalization of a close to identity
  map.
\newblock {\em ArXiv}, 2022.

\bibitem[BT19]{BT19}
P.~Berger and D.~Turaev.
\newblock On herman's positive entropy conjecture.
\newblock {\em Advances in Mathematics}, 349:1234--1288, 2019.

\bibitem[BTM20]{Burby}
J.~W. Burby, Q.~Tang, and R.~Maulik.
\newblock Fast neural poincar{\'e} maps for toroidal magnetic fields.
\newblock {\em Plasma Physics and Controlled Fusion}, 63(2):024001, 2020.

\bibitem[FK22]{FK21}
F.~Forstneri{\v{c}} and F.~Kutzschebauch.
\newblock The first thirty years of andersen-lempert theory.
\newblock {\em Analysis Mathematica}, pages 1--56, 2022.

\bibitem[GT10]{GT10}
V.~Gelfreich and D.~Turaev.
\newblock Universal dynamics in a neighborhood of a generic elliptic periodic
  point.
\newblock {\em Regular and Chaotic Dynamics}, 15(2):159--164, 2010.

\bibitem[GT17]{GT17}
S.~Gonchenko and D.~Turaev.
\newblock On three types of dynamics and the notion of attractor.
\newblock {\em Proceedings of the Steklov Institute of Mathematics},
  297(1):116--137, 2017.

\bibitem[GTS07]{GST07}
S.~Gonchenko, D.~Turaev, and L.~Shilnikov.
\newblock Homoclinic tangencies of arbitrarily high orders in conservative and
  dissipative two-dimensional maps.
\newblock {\em Nonlinearity}, 20(2):241--275, jan 2007.

\bibitem[JZZ{\etalchar{+}}20]{Jin}
P.~Jin, Z.~Zhang, A.~Zhu, Y.~Tang, and G.~Karniadakis.
\newblock Sympnets: Intrinsic structure-preserving symplectic networks for
  identifying hamiltonian systems.
\newblock {\em Neural Networks}, 132:166--179, 2020.

\bibitem[Tur02]{TUR02}
D.~Turaev.
\newblock Polynomial approximations of symplectic dynamics and richness of
  chaos in non-hyperbolic area-preserving maps.
\newblock {\em Nonlinearity}, 16(1):123, 2002.

\bibitem[Tur15]{Tu15}
D.~Turaev.
\newblock Maps close to identity and universal maps in the {N}ewhouse domain.
\newblock {\em Comm. Math. Phys.}, 335(3):1235--1277, 2015.

\bibitem[VWT{\etalchar{+}}22]{valperga2022learning}
R.~Valperga, K.~Webster, D.~Turaev, V.~Klein, and J.~Lamb.
\newblock Learning reversible symplectic dynamics.
\newblock In {\em Learning for Dynamics and Control Conference}, pages
  906--916. PMLR, 2022.

\end{thebibliography}

\end{document}